\documentclass{amsart}

\usepackage{amsmath,amssymb,yhmath, url,color,mathrsfs,mathtools}

\newtheorem{thm}{Theorem}
\newtheorem{lemma}[thm]{Lemma}
\newtheorem{prop}[thm]{Proposition}

\numberwithin{thm}{section}

\newtheorem{thm?}[thm]{Theorem?}

\theoremstyle{remark}
\newtheorem{remark}{Remark}

\begin{document}
\title{The Truth About Torsion In The CM Case, II}
\author{Pete L. Clark}
\author{Paul Pollack}

%\address{Department of Mathematics \\ Boyd Graduate Studies Research Center \\ %University
%of Georgia \\ Athens, GA 30602-7403 \\ USA}
%\email{pete@math.uga.edu}
%\email{bcook@math.ubc.ca}
%\email{stankewicz@gmail.com}

%

\newcommand{\ratl}{\mathrm{ratl}}
\newcommand{\etalchar}[1]{$^{#1}$}
\newcommand{\F}{\mathbb{F}}
\newcommand{\et}{\textrm{\'et}}
\newcommand{\ra}{\ensuremath{\rightarrow}}
\newcommand{\FF}{\F}
\newcommand{\ff}{\mathfrak{f}}
\newcommand{\Z}{\mathbb{Z}}
\newcommand{\N}{\mathbb{N}}
\newcommand{\ch}{}
\newcommand{\R}{\mathbb{R}}
\newcommand{\PP}{\mathbb{P}}
\newcommand{\pp}{\mathfrak{p}}
\newcommand{\C}{\mathbb{C}}
\newcommand{\Q}{\mathbb{Q}}
\newcommand{\ab}{\operatorname{ab}}
\newcommand{\Aut}{\operatorname{Aut}}
\newcommand{\gk}{\mathfrak{g}_K}
\newcommand{\gq}{\mathfrak{g}_{\Q}}
\newcommand{\OQ}{\overline{\Q}}
\newcommand{\Out}{\operatorname{Out}}
\newcommand{\End}{\operatorname{End}}
\newcommand{\Gal}{\operatorname{Gal}}
\newcommand{\CT}{(\mathcal{C},\mathcal{T})}
\newcommand{\lcm}{\operatorname{lcm}}
\newcommand{\Div}{\operatorname{Div}}
\newcommand{\OO}{\mathcal{O}}
\newcommand{\rank}{\operatorname{rank}}
\newcommand{\tors}{\operatorname{tors}}
\newcommand{\IM}{\operatorname{IM}}
\newcommand{\CM}{\mathrm{CM}}
\newcommand{\HS}{\mathbf{HS}}
\newcommand{\Frac}{\operatorname{Frac}}
\newcommand{\Pic}{\operatorname{Pic}}
\newcommand{\coker}{\operatorname{coker}}
\newcommand{\Cl}{\operatorname{Cl}}
\newcommand{\loc}{\operatorname{loc}}
\newcommand{\GL}{\operatorname{GL}}
\newcommand{\PSL}{\operatorname{PSL}}
\newcommand{\Frob}{\operatorname{Frob}}
\newcommand{\Hom}{\operatorname{Hom}}
\newcommand{\Coker}{\operatorname{\coker}}
\newcommand{\Ker}{\ker}
\newcommand{\g}{\mathfrak{g}}
\newcommand{\sep}{\operatorname{sep}}
\newcommand{\new}{\operatorname{new}}
\newcommand{\Ok}{\mathcal{O}_K}
\newcommand{\ord}{\operatorname{ord}}
\newcommand{\mm}{\mathfrak{m}}
\newcommand{\Ohell}{\OO_{\ell^{\infty}}}
\newcommand{\ann}{\operatorname{ann}}
\renewcommand{\tt}{\mathfrak{t}}

\renewcommand{\aa}{\mathfrak{a}}
\newcommand\leg{\genfrac(){.4pt}{}}
\newcommand{\K}{\mathscr{K}}
\newcommand{\mfg}{\mathfrak{g}}
\newcommand{\cyc}{\operatorname{cyc}}

\begin{abstract}
Let $T_{\CM}(d)$ be the largest size of the torsion subgroup of an elliptic curve with complex multiplication (CM) defined 
over a degree $d$ number field.  Work of \cite{Breuer10} and \cite{CP15} showed $\limsup_{d \ra \infty} \frac{T_{\CM}(d)}{d \log \log d} \in (0,\infty)$. Here we show that the above limit supremum is precisely $\frac{e^{\gamma} \pi}{\sqrt{3}}$.  We also study -- in part, out of necessity -- the upper order of the size of the torsion subgroup of various restricted classes of CM elliptic curves over number fields.  
\end{abstract}

\maketitle

\tableofcontents

\section{Introduction}

\subsection{Asymptotics of torsion subgroups of elliptic curves}
Let $E_{/F}$ be an elliptic curve over a number field.  Then the torsion subgroup $E(F)[\tors]$ is finite, and it is a problem of 
fundamental interest to study its size as a function of $F$ and also of $d = [F:\Q]$.  For $d \in \Z^+$, we put 
\[ T(d) = \sup \# E(F)[\tors], \]
the supremum ranging over all elliptic curves defined over all degree $d$ number fields.  We know that $T(d) < \infty$ 
for all $d \in \Z^+$ \cite{Merel96}.  Merel's work gives an
explicit upper bound on $T(d)$, but it is more than exponential.  \\ \indent
In the other direction, it is known that $T(d)$ is \emph{not} bounded above by a linear function of $d$.  This and related bounds can be obtained by the following seemingly 
naive approach: start with \emph{any} number field $F_0$ and \emph{any} elliptic curve $E_{/F_0}$.  For $n \in \Z^+$, 
let $N_n$ be the product of the first $n$ prime numbers, and put 
\[ F_n = F_0(E[N_n]), \quad d_n = [F_n:\Q]. \]
An analysis of this ``naive approach'' was given by F. Breuer \cite{Breuer10}, who showed
\begin{equation}
\label{BREUEREQ1}
\inf_n \frac{\# E(F_n)[\tors]}{\sqrt{d_n \log \log d_n}} > 0. 
\end{equation}
An elliptic curve $E_{/F}$ has \textbf{complex multiplication (CM)} if \[\End(E) \coloneqq \End_{\overline{F}}(E) \supsetneq \Z, \] in which 
case $\End(E)$ is an order $\OO$ in an imaginary quadratic field $K$.  Moreover, if $F \supset K$ we 
have $\End_F(E) = \OO$, whereas if $F \not \supset K$ we have $\End_F(E) = \Z$.  If $E_{/F_0}$ has CM, then 
Breuer shows by the same ``naive approach'' that
\begin{equation}
\label{BREUEREQ2}
 \inf_n \frac{\# E(F_n)[\tors]}{d_n \log \log d_n} > 0, 
\end{equation}
and thus $T(d)$ is not bounded above by a linear function of $d$.  In view of these and other considerations, it is reasonable to define $T_{\CM}(d)$ as 
for $T(d)$ but restricting to CM elliptic curves only and $T_{\neg \CM}(d)$ as for $T(d)$ but restricting to elliptic curves 
\emph{without} CM.  Then it follows from (\ref{BREUEREQ1}) that 
\begin{equation*}
%\label{BREUEREQ3}
 \limsup_{d \ra \infty} \frac{T_{\neg \CM}(d)}{\sqrt{d \log \log d}} > 0 
\end{equation*}
and it follows from (\ref{BREUEREQ2}) that
\begin{equation*}
%\label{BREUEREQ4}
 \limsup_{d \ra \infty} \frac{T_{\CM}(d)}{d \log \log d} > 0. 
\end{equation*}
In a recent work \cite[Thm. 1]{CP15} we showed there is an effective $C > 0$ such that 
\[ \forall d \geq 3, \quad  T_{\CM}(d) \leq C d \log \log d \]
and thus we get an upper order result for $T_{\CM}(d)$: 
\begin{equation}
\label{CPEQ}
\limsup_{d \ra \infty} \frac{T_{\CM}(d)}{d \log \log d} \in (0,\infty). 
\end{equation}
Other statistical behavior of $T_{\CM}(d)$ was studied in \cite{BCP,BP16,MPP16}; in particular, its average order is $d/(\log{d})^{1+o(1)}$ and its normal 
order (in a slightly nonstandard sense made precise in \cite{BCP}) is bounded.   
\\ \\
In the present work we will improve upon (\ref{CPEQ}), as follows:

\begin{thm}
\label{MAINTHM} $\displaystyle \limsup_{d\to\infty} \frac{T_{\CM}(d)}{d\log\log{d}} = \frac{e^{\gamma} \pi}{\sqrt{3}}. $
\end{thm}
\noindent
(The easier, lower bound half of Theorem \ref{MAINTHM} was noted in \cite[Remark 1.10]{MPP16}.) In $\S$1.4 we will deduce Theorem \ref{MAINTHM} from results stated later in the introduction.

\begin{remark}
As mentioned above, the constant $C$ appearing in \cite[Thm. 1]{CP15} is effective.  This aspect is not addressed in Theorem \ref{MAINTHM} or anywhere in the present work.  In fact, though $C$ is effectively \emph{computable}, we did not effectively 
\emph{compute} it, and it is not a trivial matter to do so.  It is our understanding that such a computation is in progress by some of our colleagues.
\end{remark}
%\noindent
%It is also interesting to ask for the upper 
%order of $\# E(F)[\tors]$ along various natural subclasses of the class of all CM elliptic curves over all degree $d$ number fields.  We discuss 
%such refinements in the next two sections.

\subsection{Refining the truth I}
As mentioned above, it is natural to distinguish between the cases in which the CM is or is not rationally defined over the 
ground field.  In this section we concentrate on the former case: let $T_{\CM}^{\bullet}(d)$ be as for $T_{\CM}(d)$ but restricting to CM elliptic curves $E_{/F}$ for number fields $F \supset K$.  
\\ \\
%\\ \indent
%\begin{thm}
%\label{BIGTHM2}
%\label{thm:nonrational} $\displaystyle\limsup_{d\to\infty} \frac{T_{\rm CM}^{\circ}(d)}{d\log\log{d}} = 0.$
%\end{thm}
%\noindent
We will also examine the dependence of the bound on the CM field and the CM order.  Let $\K$ be a set of imaginary quadratic fields.  
We define $T_{\CM(\K)}(d)$ to be as for $T_{\CM}(d)$ but with the CM field restricted to lie in $\K$.  When $\K = \{K\}$ we 
write $T_{\CM(K)}(d)$ in place of $T_{\CM(\{K\})}(d)$.  Once again we denote restriction to number fields $F \supset K$ 
by a superscripted $\bullet$.
%As above, we denote restriction to number fields $F \supset K$ 
%by a superscripted $\bullet$ and restriction to number fields $F \not \supset K$ by a superscripted $\circ$.  

\begin{thm}\label{lem:reducetofinite}
\label{BIGTHM3} Let $\epsilon > 0$. There is $\Delta_0 = \Delta_0(\epsilon) < 0$ such that: if $\K$ is the collection of imaginary quadratic fields with $\Delta_K < \Delta_0$, then
\begin{equation*}
%\label{PETEEQ0}
 \limsup_{d\to\infty} \frac{T_{\CM(\K)}^{\bullet}(d)}{d\log\log{d}} < \epsilon. 
\end{equation*}
\end{thm}
\noindent
We will prove Theorem \ref{BIGTHM3} in $\S$3.  A key ingredient is a lower order result for Euler's totient 
function across all imaginary quadratic fields which improves upon \cite[Thm. 8]{CP15}.  This result is established 
in $\S$2.  
\\ \\
Theorem \ref{BIGTHM3} motivates us to concentrate on a fixed imaginary quadratic field as well as a fixed imaginary 
quadratic order.  The next two results address this.

\begin{thm}\label{lem:fixedK} \label{BIGTHM4}  Fix an imaginary quadratic field $K$.  For $d \in \Z^+$, let 
$\mathfrak{f}T_{\CM(K)}^{\bullet}(d)$ denote the maximum value of $\mathfrak{f} \# E(F)[\tors]$ as 
$F$ ranges over all degree $d$ number fields containing $K$ and $E_{/F}$ ranges over all elliptic curves 
with CM by an order $\OO$ of $K$: here $\ff$ is the conductor of $\OO$.  Then we have 

\[ \limsup_{d\to\infty} \frac{\mathfrak{f} T_{\CM(K)}^{\bullet}(d)}{d\log\log{d}} \leq \frac{e^{\gamma} \pi}{\sqrt{|\Delta_K|}}. \]
\end{thm}

\begin{thm}
\label{BIGTHM5}
Let $\OO$ be an order in the imaginary quadratic field $K$, with conductor $\ff$ and discriminant $\Delta = \ff^2 \Delta_K$.  Let $T_{\OO\text{-}\CM}(d)$ be the maximum value of $\#E(F)[\tors]$ as $F$ ranges over all degree $d$ number 
fields containing $K$ and $E_{/F}$ ranges over all $\OO$-CM elliptic curves.  Then 
\[ \limsup_{d \ra \infty} \frac{ T_{\OO\text{-}\CM}^{\bullet}(d)}{d \log \log d} = \frac{e^{\gamma} \pi}{\sqrt{|\Delta|}} = \frac{e^{\gamma} \pi}{\ff \sqrt{|\Delta_K|}}. \]
\end{thm}
\noindent
We will prove Theorems \ref{BIGTHM4} and \ref{BIGTHM5} in $\S$4.

\subsection{Refining the truth II} 
We turn now to upper order results for $\#E(F)[\tors]$ when the CM is \emph{not} defined over the ground field $F$: define $T_{\CM}^{\circ}(d)$ as for $T_{\CM}(d)$ but restricting to CM elliptic curves $E_{/F}$ for number fields $F \not \supset K$.  As above, we will want 
to impose this restriction along with restrictions on the CM field and CM order, and we denote restriction to number fields 
$F \not \supset K$ by a superscripted $\circ$.  
\\ \indent
If $E_{/F}$ is an elliptic curve defined over a number 
field $F$ not containing $K$, then $\# E(F)[\tors] \leq \# E(FK)[\tors]$, and thus we have
\begin{equation}
\label{OBVIOUSEQ}
 T_{\CM}^{\circ}(d) \leq T_{\CM}^\bullet(2d).
\end{equation}
\noindent
Although (\ref{OBVIOUSEQ}) will be of use to us, it is too crude to allow us to deduce Theorem \ref{MAINTHM} 
from the results of the previous section.  To overcome this we establish the following result, which \emph{almost} 
computes the true upper order of $T_{\OO\text{-}\CM}^{\circ}(d)$.

\begin{thm}
\label{BIGTHM6}
Let $\OO$ be an order in an imaginary quadratic field $K$.
\begin{enumerate}
	\item[a)] There is a constant $C(\OO)$ such that 
\[\forall d \geq 3, \quad T_{\OO\text{-}\CM}^{\circ}(d) \leq C(\OO) \sqrt{d \log \log d}. \]
\item[b)] We have $\limsup_{d \ra \infty} \frac{T_{\OO\text{-}\CM}^{\circ}(d)}{\sqrt{d}} > 0$.
\end{enumerate}
\end{thm}
\noindent
We will prove Theorem \ref{BIGTHM6} in $\S$5.
\\ \\
We can now easily deduce:

\begin{thm}
\label{BIGTHM2}
\label{thm:nonrational} We have $\displaystyle\limsup_{d\to\infty} \frac{T_{\rm CM}^{\circ}(d)}{d\log\log{d}} = 0.$
\end{thm}
\begin{proof}
Step 1: Using (\ref{OBVIOUSEQ}) we immediately get versions of Theorems \ref{BIGTHM3} and \ref{BIGTHM4} for 
$T_{\CM}(d)$: since replacing $\epsilon$ by $2 \epsilon$ is harmless, 
Theorem \ref{BIGTHM3} holds verbatim with $T_{\CM}(d)$ in place of $T_{\CM}^{\bullet}(d)$, whereas for 
any imaginary quadratic field $K$ we have 
\begin{equation}
\label{BIGTHM2EQ1}
 \limsup_{d\to\infty} \frac{\mathfrak{f} T_{\CM(K)}(d)}{d\log\log{d}} \leq \frac{2 e^{\gamma} \pi}{\sqrt{|\Delta_K|}}. 
\end{equation}
Step 2: The above strengthened version of Theorem \ref{BIGTHM3} reduces us to finitely many quadratic fields, and then 
the dependence on the conductor in (\ref{BIGTHM2EQ1}) reduces us to finitely many quadratic orders.  Thus we may 
treat one quadratic order at a time, and Theorem \ref{BIGTHM6} gives a much better bound than $o(d \log \log d)$ in that case.
\end{proof}

\subsection{Proof of Theorem \ref{MAINTHM}} Theorem \ref{MAINTHM} is a quick consequence of these refined results: by Theorem \ref{BIGTHM2} 
we may restrict to the case in which the number field contains the CM field.  Now we argue much as in the proof of 
Theorem \ref{BIGTHM2}.  By Theorem \ref{BIGTHM5}, applied with $\OO$ the maximal order in $\Q(\sqrt{-3})$, 
\[ \limsup_{d\to\infty} \frac{T_{\CM}(d)}{d\log\log d} \ge \frac{e^{\gamma} \pi}{\sqrt{3}}.\]
By Theorem \ref{BIGTHM3}, in any sequence $\{(E_n)_{/F_n}\}$ with $[F_n:\Q] \ra \infty$ such that 
\[ \lim_{n \ra \infty} \frac{ \#E(F_n)[\tors]}{[F_n:\Q] \log \log [F_n:\Q]} \geq \frac{e^{\gamma} \pi}{\sqrt{3}}, \]
 only finitely many quadratic fields intervene, and by Theorem \ref{BIGTHM4} 
among orders with the same fraction field the conductors must be bounded.  So we have reduced to working with 
one imaginary quadratic order $\OO$ at a time, and Theorem \ref{BIGTHM5} tells us that $T_{\OO\text{-}\CM}^{\bullet}(d)$ 
is largest when the discriminant of $\OO$ is smallest, i.e., when $\OO$ is the ring of integers of $\Q(\sqrt{-3})$.  

\subsection{Complements} In $\S$6.1 we compare our results to the asymptotic behavior of prime order torsion studied in 
\cite{TORS1}.  In $\S$6.2 we address -- but do not completely resolve -- the question of the upper order of $T_{\CM}^{\circ}(d)$.  
In $\S$6.3 we revisit Breuer's work and give what is in a sense a non-CM analogue of Theorem \ref{BIGTHM5}: we study the asymptotic behavior of torsion one $j$-invariant at a time.

%\subsection{Structure of the paper} In broadest outline, the proof of \cite[Thm. 1]{CP15} goes as follows: if an elliptic 
%curve has a large torsion subgroup over $F$, it has full $N$-torsion for large $N$ over a small degree extension of $K$ \cite[Thm. 7]%{CP15}.  An extension of the First Main Theorem of Complex 
%Multiplication \cite[Thm. 5]{CP15} gives a lower bound on $[F(E[N]):\Q]$ in terms of the Euler totient function $\varphi_K(N \OO_K)$ %multiplied by the class number $h_K$ of $K$.  For an ideal $\aa$ in $\OO_K$ we prove a lower order result on $\varphi_K(\aa) h_K$ %\emph{uniformly} across imaginary quadratic fields \cite[Thm. 8]{CP15}.  Putting these ingredients 
%together gives the desired bound.  
%\\ \indent
%Here we make use of most of the above ingredients while also adding some new ones.  Again there is a key component of our %argument which is fundamentally analytic in nature.  In $\S$2 
%we will give a lower order result for $\varphi_K$ (Theorem \ref{thm:phibound}) which improves upon \cite[Thm. 8]{CP15}.
%We will prove Theorem \ref{BIGTHM3} in $\S$3 and Theorems \ref{BIGTHM4} and \ref{BIGTHM5} in $\S$4.  With 
%Theorems \ref{BIGTHM3} and \ref{BIGTHM4} in hand, we can prove Theorem \ref{BIGTHM2} by working one 
%imaginary quadratic order at a time, and we do so in $\S$5.

\section{Lower bounds on $\varphi_K(\aa)$}
\noindent
For the classical Euler totient function, it is a well-known consequence of Mertens' Theorem that \cite[Thm. 328]{HW} $$\liminf_{n\to\infty} \frac{\varphi(n)}{n/\log\log{n}} = e^{-\gamma}.$$ As in \cite{CP15}, we require analogous results for $\varphi_K(\aa)$, where $K$ is an imaginary quadratic field and $\aa$ is an ideal of $\OO_K$.

When the field $K$ is fixed, this presents no difficulty. In that case, one can argue precisely as in \cite{HW}, using the number field analogue of the classical Mertens Theorem: e.g. \cite{Rosen99}.  For any fixed number field $K$, let $\alpha_K$ be 
the residue of the Dedekind zeta function $\zeta_K(s)$ at $s = 1$.  Then one finds that
\[ \liminf_{|\aa|\to\infty} \frac{ \varphi_K(\aa)}{|\aa|/\log \log |\aa|} = e^{-\gamma} \alpha_K^{-1}. \]
(Here and below, $|\mathfrak{a}|$ denotes the norm of the ideal $\mathfrak{a}$.)
 When $K$ is imaginary quadratic, the class number formula  (e.g. \cite[Theorem 61, p. 284]{FT93}) gives \[\alpha_K = \frac{2\pi h_K}{w_K \sqrt{|\Delta_K|}}, \]
so that 
\begin{equation}\label{eq:landauK}
 \liminf_{|\aa|\to\infty} \frac{ \varphi_K(\aa)}{|\aa|/\log \log |\aa|} = \frac{e^{-\gamma} w_K \sqrt{|\Delta_K|}}{2\pi h_K}.\end{equation}
When $K$ is allowed to vary, the situation becoms more delicate.  In this section we will establish the following result.
\begin{thm}\label{thm:phibound}
There is a constant $c>0$ such that for all quadratic fields $K$ and all nonzero
ideals $\aa$ of $\OO_K$ with $|\aa| \geq 3$, we have
\[ \varphi_K(\aa) \geq \frac{c}{\log|\Delta_K|}\cdot \frac{|\aa|}{\log\log|\aa|}. \]
\end{thm}

%\begin{thm}\label{thm:phibound}
%There is a constant $c>0$ such that for all quadratic fields $K$ and all nonzero
%ideals $\aa$ of $\OO_K$ with $|\aa| \geq 3$, we have
%\[ \varphi_K(\aa) \geq \frac{c}{\log|\Delta_K|}\cdot \frac{|\aa|}{\log\log|\aa|}. \]
%where the implied constant is absolute.
%\end{thm}
\noindent
%The proof relies on two elementary lemmas.
 Below, $\star$ denotes Dirichlet convolution, so that $(f\star g)(n) = \sum_{de=n} f(d) g(e)$.

\begin{lemma}\label{lemma:sylvester} There is a positive constant $C$ for which the following holds. Let $f$ be a nonnegative multiplicative function. Let $\chi$ be a nonprincipal Dirichlet character modulo $q$. For all $x \geq 2$,
\begin{equation}
\label{SYLVESTEREQ}
 \left|\sum_{n \le x} \frac{(f\star\chi)(n)}{n}\right| \leq C\log{q} \cdot \prod_{p \le x} \left(1+\frac{f(p)}{p} + \frac{f(p^2)}{p^2} + \dots\right).
\end{equation}
%The implied constant is absolute.
\end{lemma}
\begin{proof} Write $\sum_{n \le x} \frac{(f\star\chi)(n)}{n} = \sum_{d \le x} \frac{f(d)}{d} \sum_{e \le x/d} \frac{\chi(e)}{e}$. The contribution to the inner sum from values of $e\le q$ is bounded in absolute value by $1+1/2+\dots+1/q\ll \log{q}$. Since the partial sums of $\chi$ are (crudely) bounded by $q$, Abel summation shows that the contribution to the inner sum from values of $e$ with $q < e \le x/d$ is $\ll 1 \ll \log{q}$. Now the triangle inequality and the nonnegativity of $f$ yield
\[ \left|\sum_{n \le x} \frac{(f\star\chi)(n)}{n}\right| \ll \log{q} \cdot \sum_{d \le x} \frac{f(d)}{d}. \]
The sum on $d$ is bounded by the Euler product appearing in (\ref{SYLVESTEREQ}).
\end{proof}

\begin{lemma}\label{lemma:lowerprod} There is a constant $c>0$ for which the following holds. Let $\chi$ be a quadratic character modulo $q$. For all $x \ge 2$,
\[ \prod_{p \le x}\left(1-\frac{\chi(p)}{p}\right) \geq \frac{c}{\log{q}}. \]
%where the implied constant is absolute.
\end{lemma}
\begin{proof} Let $f$ be multiplicative such that $f(p^k) = 1-\chi(p)$ for every prime power $p^k$. Since $\chi$ is quadratic, $f$ assumes only nonnegative values. Moreover we have
$(f \star \chi)(n) = 1$ for all $n \in \Z^+$, so that
\[ \sum_{n \le x} \frac{(f \star \chi)(n)}{n} = \sum_{n \le x}\frac{1}{n} \gg \log{x}. \] Hence, by Lemma \ref{lemma:sylvester} and Mertens' Theorem \cite[Thm. 429]{HW},
\begin{align*} \log{x} \ll (\log{q}) \cdot \prod_{p \le x} \left(1 + \frac{1-\chi(p)}{p-1}\right) &= (\log{q})\cdot \prod_{p \le x}\left(1-\frac{1}{p}\right)^{-1} \cdot \prod_{p \le x}\left(1-\frac{\chi(p)}{p}\right)
\\&\ll (\log{q}) (\log{x}) \cdot \prod_{p \le x}\left(1-\frac{\chi(p)}{p}\right). \end{align*}
Rearranging gives the lemma.
\end{proof}

\begin{proof}[Proof of Theorem \ref{thm:phibound}] We write $\varphi_K(\aa) = |\aa| \cdot \prod_{\pp \mid \aa} (1-1/|\pp|)$ and notice that the factors $1-1/|\pp|$ are increasing in $|\pp|$. So if $z\ge 2$ is such that $\prod_{|\pp| \le z} |\pp| \ge |\aa|$, then
\begin{equation}\label{eq:step0} \frac{\varphi_K(\aa)}{|\aa|} \ge \prod_{|\pp| \le z} \left(1-\frac{1}{|\pp|}\right). \end{equation}
We first establish a lower bound on the right-hand side, as a function of $z$, and then we prove the theorem by making a convenient choice of $z$.  We partition the prime ideals with $|\pp|\le z$ according to the splitting behavior of the rational prime $p$ lying below $\pp$. Noting that $p \le |\pp|$, Mertens' Theorem and Lemma \ref{lemma:lowerprod} yield
\begin{align} \prod_{|\pp| \le z} \left(1-\frac{1}{|\pp|}\right) &\ge \prod_{p \le z} \prod_{\pp \mid (p)} \left(1-\frac{1}{|\pp|}\right)  
= \prod_{p \le z} \left(1-\frac{1}{p}\right) \left(1-\frac{\leg{\Delta_K}{p}}{p}\right) \notag\\
&\gg (\log{z})^{-1} \prod_{p \le z} \left(1-\frac{\leg{\Delta_K}{p}}{p}\right) \gg (\log{z})^{-1} \cdot (\log{|\Delta_K|})^{-1}. \label{eq:step1}\end{align}
With $C'$ a large absolute constant to be described momentarily, we set
\begin{equation}\label{eq:step2} z=(C'\log |\aa|)^2.\end{equation}

Let us check that $\prod_{|\pp| \le z}|\pp| \ge |\aa|$ with this choice of $z$. In fact, the Prime Number Theorem guarantees that \[ \prod_{|\pp| \le z}|\pp| \ge \prod_{p \le z^{1/2}} p \ge \prod_{p \le C'\log |\aa|}p \ge |\aa|,\] provided that $C'$ was chosen appropriately.  Combining \eqref{eq:step0}, \eqref{eq:step1}, and \eqref{eq:step2} gives \[\varphi_K(\aa)\gg |\aa| \cdot (\log{z})^{-1} \cdot (\log |\Delta_K|)^{-1}\gg (\log|\Delta_K|)^{-1} \cdot |\aa| \cdot (\log\log|\aa|)^{-1}. \qedhere \]
%as desired.
\end{proof}

\section{Proof of Theorem \ref{BIGTHM3}}
\noindent
Let $E$ be an elliptic curve over a degree $d$ number field $F$.  Certainly we may, and shall, assume that $\# E(F)[{\rm tors}] \geq 3$.  
Suppose that $E$ has CM by the imaginary quadratic field $K$ and that $F \supset K$.   Let $\OO = \End E$, an order in the imaginary quadratic field $K$.  Let $\OO_K$ be the 
ring of integers of $K$.  By \cite[Thm. 1.3]{BC17} there is an elliptic curve $E'_{/F}$ with $\End E' = \OO_K$ such that 
\[ \# E(F)[\tors] \mid \# E'(F)[\tors].  \]
%Thus in order to get an upper bound on $T_{\CM(\K)}^{K}(d)$ we may assume $\OO = \OO_K$. 
 Let $\aa \subset \OO_K$
be the annihilator ideal of the $\OO_K$-module $E'(F)[{\rm tors}]$.  By \cite[Thm. 2.7]{BC17} we have 
\[ E'(F)[{\rm tors}] \cong \OO_K/\aa, \]
so
\[ \#E'(F)[{\rm tors}] = |\aa|. \]
 By the First Main Theorem of Complex Multiplication \cite[Thm. II.5.6]{SilvermanII} we have 
\begin{equation}
\label{PETEEQ1}
F \supset K^{\aa},
\end{equation}
while by \cite[Lemma 2.11]{BC17} we have 
\begin{equation}
\label{PETEEQ2}
  \frac{2 \varphi_K(\aa) h_K}{w_K} \mid [K^{\aa}:\Q] 
%\mid \frac{2 \varphi_K(\aa) h_K}{[U(K):U_{\aa}(K)]} = [K^{\aa}:\Q]
\end{equation}
with equality in most cases (e.g. unless $\aa \mid 6 \OO_K$).  Combining (\ref{PETEEQ1}) and (\ref{PETEEQ2}), we get
\begin{equation*}
%\label{PETEEQ3}
 \varphi_K(\aa) \mid \frac{w_K}{2} \cdot \frac{d}{h_K}. 
\end{equation*}
Combining the last inequality with the result of Theorem \ref{thm:phibound}, we get
\[ |\aa|/\log\log|\aa| \le \frac{w_K}{2c} \cdot \frac{\log|\Delta_K|}{h_K} \cdot d. \]
Siegel's Theorem (see \cite[Chapter 21]{Davenport00}) implies that $h_K \gg |\Delta_K|^{1/3}$.  So we can choose $\Delta_0$ sufficiently large and negative so that when $\Delta_K < \Delta_0$, we have
\[ \frac{\log|\Delta_K|}{h_K} \le \frac{\epsilon}{6} c. \]
Working under this assumption on $\Delta_K$,  we have
\[ |\aa|/\log\log|\aa| \le \frac{\epsilon}{2} d. \]
For all $d$ sufficiently large in terms of $\epsilon$, this implies that
\[ \#E(F)[\tors] \leq \#E'(F)[\tors] = |\aa| < \epsilon d\log\log{d}.\]
Thus,
\begin{equation*}
%\label{PETEEQ4}
 \limsup_{d \ra \infty} \frac{T_{\CM(\K)}^\bullet(d)}{d \log \log d} < \epsilon. 
\end{equation*}

%\medskip

%\noindent Step 2: If $F \not \supset K$, passing from $F$ to $FK$ gives the crude upper bound
%\begin{equation}
%\label{OBVIOUSEQ}
% T_{\CM(\K)}^{\circ}(d) \leq T_{\CM(\K)}^\bullet(2d).
%\end{equation}
%After englarging $\Delta_0$ if necessary, (\ref{OBVIOUSEQ}) and (\ref{PETEEQ4}) yield (\ref{PETEEQ0}).

\section{Proofs of Theorems \ref{BIGTHM4} and \ref{BIGTHM5}}
\noindent
In this section we fix an imaginary quadratic field $K$.

\subsection{Proof of Theorem \ref{BIGTHM4}}
Let $F \supset K$ be a number field of degree $d$, and let $E_{/F}$ be an elliptic curve with CM by an order 
$\OO$, of conductor $\ff$, in $K$.  Put $w = \# \OO^{\times}$.  We may and shall assume that $\# E(F)[\tors] \geq 5$.  We may write 
\[ E(F)[\tors] \cong \Z/a\Z \times \Z/ab \Z \]
for $a,b \in \Z^+$.  Since $\# E(F)[\tors] \geq 5$, we have $ab \geq 3$.   By \cite[Thm. 7]{CP15} there is a number field $L \supset F$ such that $F(E[ab]) \subset L$ 
and $[L:F] \leq b$.  Let $\mathfrak{h}\colon E \ra E/\Aut(E)$ be the Weber function for $E$, and put
\[ W(N,\OO) = K(\ff)(\mathfrak{h}(E[N])). \]
By \cite[Thm. 2.12]{BC17} we have $L \supset W(ab,\OO)$ and thus 
\begin{equation}
\label{THM4EQ1}
[L:\Q] \geq [W(ab,\OO):K^{(1)}][K^{(1)}:\Q] = 2h_K [W(ab,\OO):K(\ff)][K(\ff):K^{(1)}]. 
\end{equation}
We recall several formulas from \cite{BC17}: namely, we have \cite[Thm.4.8]{BC17}
\begin{equation}
\label{THM4EQ2}
 [W(ab,\OO):K(\ff)] = \frac{ (\OO/ab\OO)^{\times}}{w}
\end{equation}
and \cite[Lemma 2.3]{BC17}
\begin{equation}
\label{THM4EQ3}
 \# (\OO/ab\OO)^{\times} [K(\ff):K^{(1)}] = \varphi_K(ab\ff) \left( \frac{\varphi(N)}{\varphi(N \ff)} \right) \frac{w}{w_K}. \end{equation}
Combining (\ref{THM4EQ1}), (\ref{THM4EQ2}) and (\ref{THM4EQ3}) gives 
\[ [L:\Q] \geq \frac{2h_K}{w_K} \left( \frac{\varphi(ab)}{\varphi(ab\ff)} \right) \varphi_K(ab \ff), \]
and thus 
\[ d = [F:\Q] \geq \frac{[L:\Q]}{b} \geq \frac{2h_K}{b w_K} \left( \frac{\varphi(ab)}{\varphi(ab\ff)} \right) \varphi_K(ab \ff). \]
Multiplying by $a^2b^2 \ff^2 = |ab\ff \OO_K|$ and rearranging, we get 
\begin{equation}
\label{THM4EQ4}
\ff^2 \# E(F)[\tors] = \ff^2 a^2 b \leq \frac{w_K}{2} \frac{d}{h_K}  \frac{|abf \OO_K|}{\varphi_K(ab\ff)} 
 \frac{\varphi(ab\ff)}{\varphi(ab)}.
\end{equation}
We have $\frac{\varphi(ab\ff)}{\varphi(ab)} \leq \ff$; replacing the factor of $\frac{\varphi(ab\ff)}{\varphi(ab)}$ with $\ff$ 
in the right-hand side of (\ref{THM4EQ4}) and cancelling the $\ff$'s, we get 
\[ \ff \# E(F)[\tors] \leq \frac{w_K}{2} \frac{d}{h_K} \frac{|ab\ff\OO_K|}{\varphi_K(ab\ff)}. \]

Now using (\ref{eq:landauK}) we get that, as $\ff \# E(F)[\tors] \ra \infty$, 
\[ \ff \# E(F)[\tors] \leq (1+o(1)) \frac{e^{\gamma} \pi}{\sqrt{|\Delta_K|}} d \log \log(a^2 b^2 \ff^2). \]
Since $\ff \#E(F)[\tors] = \ff a^2 b \le a^2 b^2 \ff^2 \le (\ff \#E(F)[\tors])^2$,
\[ \log\log(a^2 b^2 \ff^2) = \log\log(\ff\#E(F)[\tors])  + O(1), \]
so that  
\[ \ff \# E(F)[\tors] \le (1 + o(1)) \frac{e^{\gamma} \pi}{\sqrt{|\Delta_K|}} d \log \log(\ff \# E(F)[\tors]). \]
Thus, 
\[ \frac{\ff \# E(F)[\tors]}{\log\log \ff \#E(F)[\tors]} \le (1 + o(1)) \frac{e^{\gamma} \pi}{\sqrt{|\Delta_K|}} d, \]
which implies that
\[ \ff \# E(F)[\tors] \le (1+o(1)) \frac{e^{\gamma} \pi}{\sqrt{|\Delta_K|}} d\log\log{d}. \]
Thus, for any sequence of $K$-CM elliptic curves $E_{/F}$ (having $F\supset K$, and $[F:\Q]=d$) with $\ff \#E(F)[\tors]\to\infty$,
\[ \limsup \frac{\ff \# E(F)[\tors]}{d \log \log d} \leq \frac{e^{\gamma} \pi}{\sqrt{|\Delta_K|}}.\]
Theorem \ref{BIGTHM4} follows immediately.

\subsection{Proof of Theorem \ref{BIGTHM5}} Let $\OO$ be the order of conductor $\ff$ in the imaginary 
quadratic field $K$.  Again we put $w = \# \OO^{\times}$.
\\ \\
The inequality 
\[ \limsup_{d \ra \infty} \frac{T_{\OO-\CM}^{\bullet}(d)}{d \log \log d} \leq \frac{e^{\gamma} \pi}{\ff \sqrt{|\Delta_K|}} \]
is immediate from Theorem \ref{BIGTHM4}, so it remains to prove the opposite inequality.
\\ \indent
Let $n \geq 2$, and let $N_n$ be the product of the primes not exceeding $n$.  We assume that $n$ is large enough so that for all primes $\ell$, 
if $\ell \mid \ff$ then $\ell \mid N_n$.  By \cite[Thm. 1.1b)]{BC17} there is a number field $F \supset K$ and an $\OO$-CM elliptic curve $E_{/F}$ such that \[[F:K(j(E))] = 
\frac{ \# (\OO/N_n\OO)^{\times}}{w} \]
and $(\Z/N_n\Z)^2 \hookrightarrow E(F)$.  Since $K(j(E)) = K(\ff)$, using (\ref{THM4EQ3}) as above shows that
\[ [F:\Q] = \frac{2 h_K}{w_K} \left( \frac{\varphi(N_n)}{\varphi(N_n\ff)} \right) \varphi_K(N_n\ff). \]
Because every prime dividing $\ff$ also divides $N_n$, we have $ \frac{\varphi(N_n)}{\varphi(N_n\ff)} = \frac{1}{\ff}$, 
so
\[ d = [F:\Q] = \frac{2h_K}{\ff w_K} \varphi_K(N_n \ff) = \frac{2 \ff h_K N_n^2}{w_K} \prod_{p \leq n} \bigg(1-\frac{1}{p}\bigg)\bigg(1- \leg{\Delta_K}{p} \frac{1}{p} \bigg). \]
It follows that 
\[ \# E(F)[\tors] \geq N_n^2 = \frac{w_K}{2 \ff h_K} d \prod_{p \leq n} \bigg(1-\frac{1}{p} \bigg)^{-1} \prod_{p \leq n}
\bigg(1-\leg{\Delta_K}{p}  \frac{1}{p} \bigg)^{-1}. \]
Mertens' Theorem gives $\prod_{p \le n}(1-1/p)^{-1} \sim e^{\gamma}\log{n}$, as $n\to\infty$, while 
\[\prod_{p \le n}\left(1-\leg{\Delta_K}{p}\frac{1}{p}\right)^{-1} \to L(1,\leg{\Delta_K}{\cdot}) = \frac{2 \pi h_K}{w_K \sqrt{|\Delta_K|}}. \]
Thus we find that as $n \ra \infty$ we have
\[ \# E(F)[\tors] \geq (1 + o(1)) \frac{e^{\gamma} \pi}{\ff \sqrt{|\Delta_K|}} d \log n. \]
Moreover, for sufficiently large $n$ we have
\[ d = \frac{2 h_K}{\ff w_K} \varphi_K(N_n \ff) = \frac{2 \ff h_K}{w_K} \varphi_K(N_n) \leq \frac{2 \ff h_K}{w_K} 
N_n^2 \leq N_n^3, \]
so as $n \ra \infty$ we have 
\[ \log \log d \leq \log \log N_n + O(1). \]
By the Prime Number Theorem we have $N_n = e^{(1+o(1))n}$, so 
\[ \log \log N_n \sim \log n, \]
and thus as $n \ra \infty$ we have
\[ \log n \geq (1+o(1)) \log \log d. \]
We conclude that as $n \ra \infty$, 
\[ \# E(F)[\tors] \geq (1+o(1)) \frac{e^{\gamma} \pi}{\ff \sqrt{|\Delta_K|}} d \log \log d = (1+o(1))
\frac{e^{\gamma} \pi}{\sqrt{|\Delta|}} d \log \log d, \]
completing the proof of Theorem \ref{BIGTHM5}.

\section{Proof of Theorem \ref{BIGTHM6}}

\subsection{Proof of Theorem \ref{BIGTHM6}a)}
Let $\OO$ be an order in the imaginary quadratic field $K$.  
%Recall from (\ref{OBVIOUSEQ}) that for all $d$ we have 
%\[ T_{\CM}^{\circ}(d) \leq T_{\CM}^{\bullet}(2d). \]
%Using this and Theorems \ref{BIGTHM3} and \ref{BIGTHM4}, it is enough to work with a fixed imaginary quadratic 
%order $\OO$, with fraction field $K$, conductor $\ff$ and discriminant $\Delta = \ff^2 \Delta_K$.  So it suffices to establish 
%the following result, which is of some independent interest.
%\begin{thm}
%Let $\OO$ be an order in $K$, of conductor $\ff$ and discriminant $\Delta = \ff^2 \Delta_K$.  There is $C(\OO) > 0$ 
%such that for all number fields $F \not \supset K$ with $d = [F:\Q] \geq 3$ and all $\OO$-CM elliptic curves $E_{/F}$ we 
%have 
%\[ \# E(F)[\tors] \leq C(\OO) \sqrt{d \log \log d}. \]
%\end{thm}
%\begin{proof}
By \cite[Prop. 25]{TORS1} there is a cyclic, degree $\ff$ $F$-rational isogeny $E \ra E'$, with $E'_{/F}$ an $\OO_K$-CM elliptic curve.  It follows that 
\begin{equation}
\label{JUSTKIDDINGEQ0}
\#E(F)[\tors] \leq \ff \# E'(F)[\tors].
\end{equation}
By \cite[Lemma 3.15]{BCS16} we have
\[ E'(F)[\tors] \cong \Z/a\Z \times \Z/ab \Z \]
with $a \in \{1,2\}$.  Certainly there are $A,B \in \Z^+$ such that
\[ E'(FK)[\tors] \cong \Z/A\Z \times \Z/AB\Z.\]
Write $ab = c_1 c_2$ with $c_1$ divisible only by primes $p \nmid \Delta_K$ and $c_2$ divisible only by 
primes $p \mid \Delta_K$.  Then \cite[Thm. 4.8]{BCS16} gives $c_1 \mid A$.  Let $\beta$ be the product of 
the distinct prime divisors of $c_2$, and let $\mathfrak{b}$ be the product of the distinct prime divisors of 
$\Delta_K$, so $\beta \mid \mathfrak{b}$.  By \cite[$\S$6.3]{BC17} we have $\frac{c_2}{\beta} \mid A$.  Since $\gcd(c_1,c_2) = 1$, this implies $\frac{ab}{\beta} \mid A$ and thus
\begin{equation}
\label{JUSTKIDDINGEQ2}
 \# E'(FK)[\tors] = A^2 B \geq A^2 \geq \frac{a^2 b^2}{\beta^2}. 
\end{equation}
%\[ b \leq \sqrt{\# E'(FK)[\tors]} \frac{\beta}{a}. \]
Using (\ref{JUSTKIDDINGEQ0}) and (\ref{JUSTKIDDINGEQ2}) we get
\[ \#E(F)[\tors] \leq \ff \# E'(F)[\tors]  \leq a \ff \beta \sqrt{ \# E'(FK)[\tors]} \leq 2 \ff \mathfrak{b} \sqrt{ \# E'(FK)[\tors]}.  \]
Note that $\ff$ and $\mathfrak{b}$ depend only on $\OO$. Moreover $\# E'(FK)[\tors] \ll_{\OO} d\log\log d$, by Theorem \ref{BIGTHM5}. Thus, as claimed, we have
\[ \#E(F)[\tors] \ll_{\OO} \sqrt{d\log\log{d}}. \] 

%\left(\frac{4 \ff \mathfrak{b} e^{\gamma} \pi}{\sqrt{|\Delta_K|}} + 
%o(1) \right) \sqrt{d \log \log d}. \qedhere \]

\subsection{Proof of Theorem \ref{BIGTHM6}b)} First one rather innocuous preliminary result.

\begin{lemma}
\label{REAL1}
Let $F$ be a subfield of $\R$.  Let $E_{/F}$ be an elliptic curve, and let 
$N \in \Z^+$ have prime power decomposition $N= \prod_{i=1}^r \ell_i^{a_i}$.  Then there is a point $P \in E(\R)$ of order $N$ such that $[F(P):F] \leq \prod_{i=1}^r \ell_i^{2a_i-2}(\ell_i^2-1)$.
\end{lemma}
\begin{proof}
As for every elliptic curve defined over $\R$, we have $E(\R) \cong S^1$ or $E(\R) \cong S^1 \times \Z/2\Z$ (e.g. \cite[Cor. V.2.3.1]{SilvermanII}) .  Thus there is $P \in E(\R)$ of order $N$.  Let $\overline{F}$ be the algebraic closure of $F$ viewed as a subfield of $\C$.  Then $P \in E(\overline{F})$ and 
the degree $[F(P):F]$ is the size of the $\Aut(\overline{F}/F)$-orbit on $P$.  For all $\sigma \in \Aut(\overline{F}/F)$, $\sigma(P)$ is also a point of order $N$, 
so the size of this orbit is no larger than the number of order $N = \prod_{i=1}^r \ell_i^{a_i}$ points in $E[N](\overline{F}) 
\cong (\Z/N\Z)^2$, which is $\prod_{i=1}^r \ell_i^{2a_i-2}(\ell_i^2-1)$.  
\end{proof}
\noindent
We now give the proof of Theorem \ref{BIGTHM6}b).  Let $\OO$ be an order in an imaginary quadratic field $K$.  Let 
$F_0 = \Q(j(\C/\OO))$, so that $F_0$ is a subfield of $\R$ (forcing $F_0\not\supset K$) and $[F_0:\Q]=\#\Pic\OO$.  Let $E_{/F_0}$ be any $\OO$-CM elliptic curve.  Let $r \in \Z^+$ and let $N_r = p_1 \cdots p_r$ be the product of the first $r$ primes. \\ \indent
Applying Lemma \ref{REAL1} to $E_{/F_0}$ we get a number field $F_{N_r} \subset \R$ with \[d_r \coloneqq [F_{N_r}:\Q] \leq \# \Pic \OO \prod_{i=1}^r (p_i^2-1)\]  such that $E(F_{N_r})$ has a point of order $N_r$.  So we have 
\[ \limsup_{r \ra \infty} \frac{d_r}{N_r^2} \leq \frac{\# \Pic \OO}{\zeta(2)} = \frac{6 \# \Pic \OO}{\pi^2} \]
and thus, as $r\to\infty$,
\[ \frac{\# E(F_{N_r})[\tors]}{\sqrt{d_r}} \geq \sqrt{\frac{\pi^2}{6\# \Pic \OO}} + o(1). \]

\section{Complements}
\noindent
For a number field $F$, let $\mfg_F = \Aut(\overline{\Q}/F)$ denote the absolute Galois group of $F$.  

\subsection{Comparison to prior work}
%It is interesting to contrast Theorem \ref{BIGTHM5} with prior results on torsion points of prime order from \cite{TORS1}.  
Fix an imaginary quadratic order $\OO$ of discriminant $\Delta$. For all sufficiently large primes $p$, the least degree of a number field $F \supset K$ such that there is an $\OO$-CM elliptic curve $E_{/F}$ with an $F$-rational point of order $p$ is at least $\left(\frac{2 \# \Pic \OO}{\# \OO^{\times}}\right) (p-1)$, with equality if $p$ splits in $\OO$, and thus the upper order of the size of a prime order torsion point divided by the degree of the number field containing $K$ over which it is defined is $\frac{\# \OO^{\times}}{2 \# \Pic \OO}$.  The maximum value of this quantity is $3$, occurring iff $\Delta = -3$; the next largest value 
is $2$, occurring iff $\Delta = -4$, and these are indeed the largest two imaginary quadratic discriminants.  But 
the next largest value is $1$, occurring iff $\Delta \in \{-7,-8,-11,-12,-16,-19,-27,-28, -43,-67,-163\}$.  In particular, both the 
class number $h_K$ and the size of the unit group $\OO^{\times}$ play a role in the asymptotic behavior of prime order 
torsion but get cancelled out by the special value $L(1,\left(\frac{\Delta_K}{\cdot}\right))$ 
when we look at the size of the torsion subgroup as a whole.

\subsection{The truth about $T_{\CM}^{\circ}(d)$?}

\begin{prop} There is a sequence of $d\to\infty$ along which $T^{\circ}_{\CM}(d) \ge d^{2/3+o(1)}$.
\end{prop}
\begin{proof} Let $\ell$ be an odd prime. By Corollary 7.5 in [BCS] applied to the maximal order of $K=\Q(\sqrt{-\ell})$, there is a number field  $F$ of degree $d=h_{\Q(\sqrt{-\ell})} \frac{\ell-1}{2}$ and an $\OO_K$-CM elliptic curve $E_{/F}$ with an $F$-rational torsion point of order $\ell$. We restrict to $\ell\equiv 3\pmod{4}$ --- this has the effect of ensuring that $d$ is odd, and so $F\not\supset K$. Hence, $T_{\CM}^{\circ}(d) \ge \ell$. By Dirichlet's class number formula together with the elementary bound $L(1,\leg{-\ell}{\cdot}) \ll \log{\ell}$ (see, e.g.,  \cite[Thm. 8.18]{tenenbaum}), we have $h_{\Q(\sqrt{-\ell})} \ll \ell^{1/2}\log{\ell}$. Thus, $d \le \ell^{3/2+o(1)}$ (as $\ell\to\infty$), and so $$T_{\CM}(d) \ge \ell \ge d^{2/3+o(1)}.$$ As $\ell$ tends to infinity, so does $d$, and the proposition follows.
\end{proof}
\noindent
Define $T_{\CM,\mathrm{max}}^{\circ}(d)$ in the same way as $T_{\CM}^{\circ}(d)$, but with the added restriction that we consider only curves $E_{/F}$ with CM by the maximal order $\OO_K$. The preceding proof shows that $T_{\CM,\mathrm{max}}^{\circ}(d) \ge d^{2/3+o(1)}$ on a sequence of $d$ tending to infinity.

\begin{thm} 
\label{PAULSTHM}
For all $\epsilon > 0$, there is $C(\epsilon) > 0$ such that for all $d \in \Z^+$ we have 
\[T^{\circ}_{\CM,\mathrm{max}}(d) \leq C(\epsilon) d^{2/3+\epsilon}. \]
\end{thm}

\begin{proof} 
Let $F \not \supset K$ be a number field of degree $d$, and let $E_{/F}$ be an $\OO_K$-CM elliptic curve.
There are positive integers $a$ and $b$ with
\[ E(F)[{\tors}] \cong \Z/a\Z \times \Z/ab\Z. \]
By [BCS, Lemma 3.15], we have $a \in \{1,2\}$. The remainder of the proof takes two forms depending on the size of $|\Delta_K|$. 

\medskip

\noindent Case I: $|\Delta_K| \ge d^{2/3}$. 

\medskip
\noindent Since $E(FK)$ contains a point of order $ab$, [BC, Theorem 5.3] shows that
\[ \varphi(ab) \le \frac{w_K}{2} \frac{2d}{h_K}. \]
By Siegel's Theorem, $h_K \gg_{\epsilon} |\Delta_K|^{1/2-\epsilon} \ge d^{1/3-2\epsilon/3}$ (as $d\to\infty)$, and so 
\[ \varphi(ab) \ll_{\epsilon} d^{2/3+2\epsilon/3}. \]
Consequently,
\[ \#E(F)[{\tors}] = a(ab) \le 2ab \ll_{\epsilon} d^{2/3+\epsilon}. \]

\medskip

\noindent Case II: $|\Delta_K| < d^{2/3}$. 

\medskip

\noindent We can and will assume that $d\ge 3$ and that $\#E(F)[{\rm tors}]\ge 3$. Write $ab = c_1 c_2$, where $(c_1,\Delta_K)=1$ and where every prime dividing $c_2$ divides $\Delta_K$. By [BCS, Theorem 4.8], $E(FK)$ has full $c_1$-torsion, so that $$ c_1^2 \mid \#E(FK)[\tors].$$ Let $\ell^{\alpha}$ be a prime power dividing $c_2$, and let $P$ be a point of $E(FK)$ of order $\ell^{\alpha}$. By [BC, \S6.3], the $\OO_K$-submodule of $E(FK)[{\rm tors}]$ generated by $P$ is isomorphic, as a $\Z$-module, to either $\Z/\ell^{\alpha}\Z \oplus \Z/\ell^{\alpha}\Z$ or $\Z/\ell^{\alpha}\Z \oplus \Z/\ell^{\alpha-1}\Z$. It follows that if $r$ is the product of the distinct primes dividing $c_2$, then
\[ c_2^2 \mid r\cdot\#E(FK)[{\rm tors}].\]
Since $\gcd(c_1,c_2)=1$ and $r \mid \Delta_K$, we have
\begin{equation}\label{eq:c1c2square} c_1^2 c_2^2 \mid \Delta_K \cdot \#E(FK)[\tors]. \end{equation}
Applying the proof of Theorem \ref{BIGTHM3} to $FK$, we get that if $\aa \subset \OO_K$ is the annihilator ideal of the $\OO_K$-module $E(FK)[{\rm tors}]$, then 
we have 
\[ |\aa| = \#E(FK)[\tors], \]
and 
\[ \varphi_K(\aa) \le \frac{w_K}{2} \cdot \frac{2d}{h_K}. \]
Hence, by Theorem \ref{thm:phibound},
\[ |\aa|/\log\log|\aa| \le \frac{w_K}{c} \cdot d \frac{\log |\Delta_K|}{h_K}.  \]
Since $\frac{w_K}{c}, \frac{\log |\Delta_K|}{h_K}$ are bounded, this implies $|\aa|/\log\log|\aa| \ll d$, and hence  $|\aa| \ll d\log\log{d}$. Hence, $\log\log|\aa| \ll \log\log{d}$, and
\[ \#E(FK)[{\rm tors}]= |\aa| \ll \frac{\log|\Delta_K|}{h_K}\cdot  d\log\log{d}. \]
Now \eqref{eq:c1c2square} implies that
\[ \#E(F)[{\rm tors}]^2 = (a^2 b)^2 = a^2 \cdot a^2 b^2 \le 4 c_1^2 c_2^2 \ll \frac{|\Delta_K| \log|\Delta_K|}{h_K}\cdot  d\log\log{d}. \]
By Siegel's Theorem,  $h_K \gg_{\epsilon} |\Delta_K|^{1/2-\epsilon}$. Thus (keeping in mind our upper bound on $|\Delta_K|$ in this case), 
\[ \frac{|\Delta_K| \log|\Delta_K|}{h_K} \ll_{\epsilon} |\Delta_K|^{1/2+2\epsilon} \le d^{1/3 + 4\epsilon/3},  \]
so that
\[ \#E(F)[{\rm tors}]^2 \ll_{\epsilon} d^{4/3+2\epsilon}. \]
Hence, \[ \#E(F)[{\tors}] \ll_{\epsilon} d^{2/3+\epsilon}.\] 
The result follows from combining Cases 1 and 2.
\end{proof}
\noindent
The above results suggest to us that the upper order of $T_{\CM}^{\circ}(d)$ is $d^{2/3+o(1)}$, but we cannot yet prove this.  When the CM is $F$-rationally defined, we were able to take advantage of the recent work \cite{BC17}.  The authors of \cite{BC17} are pursuing analogous algebraic results when the CM is not rationally defined.  In view of this, we hope to revisit the upper order of $T_{\CM}^{\circ}(d)$ later and present more definitive results.
%\\ \\
%It follows from these results that too much variation across quadratic orders results in smaller torsion subgroups: more precisely, in any sequence $(E_n)_{/F_n}$ of elliptic curves over number fields 
%in which $d_n = [F_n:\Q] \ra \infty$, if \[L =\limsup_{n \ra \infty} \frac{ \# E_n(F_n)[\tors]}{d_n \log \log d_n} > 0, \]
%then there is some $\OO$ such that $\End(E_n) \cong \OO$ for infinitely many $n$, and if $\Delta$ is the largest (least negative) %discriminant which occurs infinitely often, then $L \leq \frac{ e^{\gamma}}{\sqrt{|\Delta|}}$.  

\subsection{An analogue of Theorem \ref{BIGTHM5} in the non-CM case}
Our method of showing $\limsup_{d \ra \infty} \frac{ T_{\OO\text{-}\CM}^{\bullet}(d)}{d \log \log d} \geq \frac{e^{\gamma} \pi}{\sqrt{|\Delta|}}$ is very nearly the ``naive approach'' of starting with an $\OO$-CM elliptic curve defined over $F_0 = K(\mathfrak{f})$ 
and extending the base to $\tilde{F}_n = F_0(E[N_n])$.  In fact we pass to the Weber function field $F_n = F_0(\mathfrak{h}(E[N_n]))$ and then twist $E_{/F_n}$ to get full $N_n$-torsion.  We know the degree $[F_n:\Q]$ exactly; the degree $[\tilde{F}_n:\Q]$ depends on the $F_0$-rational model, but in general is $\# \OO^{\times}$ as large, so the naive approach would give \[\limsup_{d \ra \infty} \frac{ T_{\OO\text{-}\CM}^{\bullet}(d)}{d \log \log d} \geq \frac{e^{\gamma} \pi}{\# \OO^{\times} \sqrt{|\Delta|}}. \]
So the naive approach comes within a twist of giving the \emph{true upper order} of $T_{\CM}(d)$.  
\\ \\
Observe that in the CM case, fixing the quadratic order $\OO$ fixes the $\mfg_{\Q}$-conjugacy class 
of the $j$-invariant.  This motivates the following definition: let $j \in \overline{\Q} \subset \C$ and let $F_0 = \Q(j)$.  For 
positive integers $d$ divisible by $[F_0:\Q]$, put 
\[ T_j(d) = \sup \# E(F)[\tors], \]
the supremum ranging over number fields $F \subset \C$ with $[F:\Q] = d$ and elliptic curves $E_{/F}$ such that $j(E) = j$.   Note 
that we could equivalently range over all elliptic curves $E_{/F}$ such that $j(E)$ and $j$ are $\mfg_\Q$-conjugates.
%this amounts to replacing $F$ with $\iota(F)$ for some embedding $\iota\colon F \hookrightarrow \C$.

\begin{thm}[Breuer \cite{Breuer10}]\mbox{ }
\label{BREUERTHM}
\begin{enumerate}
	\item[a)] If $j \in \overline{\Q}$ is a CM $j$-invariant, then 
\[ \limsup_d \frac{T_j(d)}{d \log \log d} \in (0,\infty). \]
\item[b)] If $j \in \overline{\Q}$ is not a CM $j$-invariant, then 
\[ \limsup_d \frac{T_j(d)}{\sqrt{d \log \log d}} \in (0,\infty). \]
\end{enumerate}
\end{thm}
\noindent
Breuer states his results for a fixed elliptic curve $E_{/F_0}$, but an immediate twisting argument gives the result for fixed $j$.
%, at the cost of multiplying the upper 
%bound by a factor of $\#\Aut E \in \{2,4,6\}$.  
\\ \indent
From this perspective our Theorem \ref{BIGTHM5} can be viewed as sharpening Theorem \ref{BREUERTHM}a) by computing the value of $\limsup_d \frac{T_j(d)}{d \log \log d}$ 
for every CM $j$-invariant.  We will now 
give an analogous sharpening of Theorem \ref{BREUERTHM}b).  For a non-CM elliptic curve $E$ defined over a number field $F$, we define the \textbf{reduced Galois representation}
\[ \overline{\rho}_N\colon \mfg_F \ra \GL_2(\Z/N\Z)/\{\pm 1\} \]
as the composite of the mod $N$ Galois representation $\rho_N\colon \mfg_F \ra \GL_2(\Z/N\Z)$ with the quotient map 
$\GL_2(\Z/N\Z) \ra \GL_2(\Z/N\Z)/\{ \pm 1\}$.  The point is that if $(E_1)_{/F}$ and ${(E_2)}_{/F}$ have $j(E_1) = j(E_2)$, 
then their reduced Galois representations are the same  (up to conjugacy in $\GL_2(\Z/N\Z)/\{\pm 1\}$).  We say that $j \in \overline{\Q}$ 
is \textbf{truly Serre} if for every $E_{/F_0}$ with $j(E) = j$, then $\overline{\rho}_N$ is surjective for all $N \in \Z^+$.  

\begin{thm}
\label{NONCMTHM}
Let $j \in \overline{\Q} \subset \C$ be a non-CM $j$-invariant, and let $F_0 = \Q(j)$. 
\begin{enumerate}
	\item[a)] We have \begin{equation}
\label{TWISTEDBREUER1}
 \limsup_{d \ra \infty} \frac{T_j(d)}{\sqrt{d \log \log d}} \geq \sqrt{\frac{\pi^2 e^{\gamma}}{3[F_0:\Q]}}. 
\end{equation}
\item[b)] If $j$ is truly Serre, then 
\begin{equation}
\label{TWISTEDBREUER2}
 \limsup_{d \ra \infty} \frac{T_j(d)}{\sqrt{d \log \log d}} = \sqrt{\frac{\pi^2 e^{\gamma}}{3[F_0:\Q]}}. 
\end{equation}
\end{enumerate}
\end{thm}
\begin{proof}
a) Let ${(E_0)}_{/F_0}$ be an elliptic curve with $j(E_0) = j$.  Let $r \geq 2$, let $N_r$ be the product of all primes $p \leq r$, and let $F_{N_r} = F_0(x(E_0[N_r]))$.  
Then 
\[ d_r \coloneqq [F_{N_r}:\Q] \mid [F_0:\Q]\frac{ \# \GL_2(\Z/N_r\Z)}{2} = [F_0:\Q]\frac{\prod_{p \leq r} (p^2-1)(p^2-p)}{2}. \]
The mod $N_r$-Galois representation on $(E_0)_{/F_{N_r}}$ has image contained in $\{ \pm 1\}$ and thus is given by 
a quadratic character $\chi\colon \Aut(\overline{F_{N_r}}/F_{N_r}) \ra \{ \pm 1\}$.  Let $(E_{N_r})_{/F_{N_r}}$ be the twist of 
$(E_0)_{/F_{N_r}}$ by $\chi$, so that \[(\Z/N_r \Z)^2 \hookrightarrow E_{N_r}(F_{N_r}) \]
and thus 
\[ \# E_{N_r}(F_{N_r})[\tors] \geq N_r^2. \]
Arguments very similar to those made above -- using the Euler product for $\zeta(2)$, Mertens' Theorem and the Prime Number 
Theorem -- give 
\[ d_r \leq (1+ o(1)) \frac{3[F_0:\Q]}{\pi^2 e^{\gamma} \log r} N_r^4\]
and 
\[ \# E(F_{N_r})[\tors]^2 \geq N_r^4 \geq (1+o(1)) \frac{\pi^2 e^{\gamma}}{3[F_0:\Q]} d_r \log \log d_r, \]
and thus (\ref{TWISTEDBREUER1}) follows. \\
b) Let $E_{/F}$ be an elliptic curve with $j(E) = j$ and $[F:\Q] = d$.  We may and shall assume that $d\ge 3$ and that $\# E(F)[\tors] \geq 5$. Write 
\[ E(F)[\tors] \cong \Z/a\Z \times \Z/ab \Z \]
for $a,b \in \Z^+$, and note that we have $ab \geq 3$. 

\medskip

\noindent Step 1: We claim that there is a number field $L \supset F$ such that $F(E[ab]) \subset L$ and $[L:F] \leq b^2$.  This is the non-CM analogue of \cite[Thm. 7]{CP15}.  As in \emph{loc. cit.} primary decomposition and induction quickly reduce us to the consideration of the case $a = p^A$, $b = p$ for a prime number $p$ and $A \in \N$, and we must show that we have full $p^{A+1}$-torsion 
in an extension of degree at most $p^2$.   If $A = 0$, then we have an $F$-rational point $P$ of 
order $p$.  If $Q \in E[p] \setminus \langle P \rangle$, then the Galois orbit on $Q$ has size at most $p^2-p$, so we may 
take $L = F(Q)$ and get $[L:F] \leq p^2-p$.  If $A \geq 1$, then $E$ has full $p^A$-torsion defined over $F$ and an $F$-rational point 
$P$ of order $p^{A+1}$.  Let $Q \in 
E[p^{A+1}]$ be such that $\langle P,Q \rangle = E[p^{A+1}]$.  For $\sigma \in \mfg_F$, we have $p \sigma(Q) = \sigma(pQ) = 
pQ$, so there is $R_{\sigma} \in E[p]$ such that 
\[ \sigma(Q) = Q + R_{\sigma}. \]
Thus there are at most $p^2$ possibilities for $\sigma(Q)$, so again we may take $L = F(Q)$. 
\medskip

\noindent 
Step 2: Let $W(N) = F_0(x(E[N]))$.  Then $L \supset W(ab)$, and thus
\[ [L:\Q] \geq [W(ab):F_0][F_0:\Q] = [F_0:\Q] \frac{\# \GL_2(\Z/ab\Z)}{2}. \]
Hence,
\[ d = [F:\Q] \geq \frac{[L:\Q]}{b^2} \geq \frac{[F_0:\Q]}{2b^2} \# \GL_2(\Z/ab\Z). \]
Multiplying by $a^4b^4$ and rearranging, we get 
\[ \# E(F)[\tors]^2 = a^4 b^2 \leq  \frac{2 d}{[F_0:\Q]} \cdot \frac{(ab)^4}{\#\GL_2(\Z/ab\Z)}. \]
Now 
\begin{align*} \frac{(ab)^4}{\# \GL_2(\Z/ab\Z)} = \prod_{p \mid ab} \frac{p^4}{(p^2-1)(p^2-p)} &= \prod_{p \mid ab} \left(1-\frac{1}{p^2}\right)^{-1} \prod_{p \mid ab}\left(1-\frac{1}{p}\right)^{-1} \\ &\le \frac{\pi^2}{6} \prod_{p \mid ab}\left(1-\frac{1}{p}\right)^{-1}. \end{align*}
Substituting this above,
\begin{equation}\label{eq:a4b2bound} \#E(F)[\tors]^2 = 	a^4 b^2 \le \frac{\pi^2}{3 [F_0:\Q]} d \prod_{p\mid ab}\left(1-\frac{1}{p}\right)^{-1}. \end{equation}
Rearranging,\[ d \ge  \frac{3[F_0:\Q]}{\pi^2} \cdot a^4 b^2 \cdot \prod_{p \mid ab}(1-1/p) = \frac{3[F_0:\Q]}{\pi^2} a^3 b \cdot \varphi(ab) > \frac{1}{2} ab.   \]
(In the last step, we used the lower bound $\varphi(ab) \ge 2$.) The product  in \eqref{eq:a4b2bound} is only increased if $p$ is taken to run over the first $\omega$ primes, where $\omega$ is the number of distinct primes dividing $ab$. Since $ab < 2d$, the first $\omega$ primes all belong to the interval $[1,2\log{d}]$, once $d$ is large enough. Hence, as $d\to\infty$,
\[ \prod_{p\mid ab}\left(1-\frac{1}{p}\right)^{-1} \le \prod_{p \le 2\log{d}} \left(1-\frac{1}{p}\right)^{-1} = (1+o(1)) e^{\gamma} \log\log{d}. \]
Plugging this back into \eqref{eq:a4b2bound} and taking square roots yields the upper bound \[T_j(d) \le \left(\sqrt{\frac{\pi^2 e^{\gamma}}{3[F_0:\Q]}}+o(1)\right)\sqrt{d\log\log{d}}. \] Combining this with the lower bound from part a), the result follows.
\end{proof}

\begin{remark}\mbox{ }
\begin{enumerate}
	\item[a)] If $E_{/F_0}$ is an elliptic curve over a number field with surjective adelic Galois representation
$\hat{\rho}\colon \mfg_{F_0} \ra \GL_2(\widehat{\Z})$,
then $j(E_0)$ is truly Serre.   The converse also holds.  Indeed, suppose $j$ is truly Serre, and let $E_{/F_0}$ be any elliptic curve 
with $j(E) = j$, let $N \in \Z^+$, and let $\rho_N\colon \mfg_{F_0} \ra \GL_2(\Z/N\Z)$ be the mod $N$ Galois representation.  By definition 
of truly Serre, we have $\langle \rho_N(\mfg_{F_0}),-1\rangle = \GL_2(\Z/N\Z)$.   It follows \cite[p. 145]{Serre-MW} that $\rho_N(\mfg_{F_0}) = \GL_2(\Z/N\Z)$.  Since this holds for all $N \in \Z^+$, it follows that $\hat{\rho}$ is surjective.
\item[b)] Greicius showed \cite[Thm. 1.2]{Greicius10} that if $E_{/F_0}$ is an elliptic curve over a number field with surjective 
adelic Galois representation, then
$F_0 \cap \Q^{\ab} = \Q$ and $\sqrt{\Delta} \notin F_0 \Q^{\ab}$, where $\Delta$ is the discriminant of any Weierstrass model of $E$.  Thus if 
$[F_0:\Q] \leq 2$ the adelic Galois representation cannot be surjective.  Greicius also exhibited an elliptic curve over a non-Galois cubic field with surjective adelic Galois representation \cite[Thm. 1.5]{Greicius10}.  Zywina showed \cite{Zywina10} 
that if $F_0 \supsetneq \Q$ is a number field such that $F_0 \cap \Q^{\ab} = \Q$ then there is an elliptic curve $E_{/F_0}$ 
with surjective adelic Galois representation.  In fact he shows that ``most Weierstrass equations over 
$F_0$'' define an elliptic curve with surjective adelic Galois representation.  His work makes it plausible that when measured by height, ``most $j \in F_0$'' are truly Serre.
%\item[c)] If $j \in \Q$, then since $j$ is not truly Serre, we expect (\ref{TWISTEDBREUER1}) not to be sharp.   By Serre's Open Image Theorem \cite{Serre72}, for every non-CM elliptic curve $E_{/F_0}$, there is an $N_0 \in \Z^+$ 
%such that for all multiples $N$ of $N_0$, the indices 
%\[I(N) = [\GL_2(\Z/N\Z)/\{\pm 1\}:\overline{\rho}_{N}(\mfg_{F_0})] \]
%and 
%\[I(N_0) = [\GL_2(\Z/N_0\Z)/\{\pm 1\}:\overline{\rho}_{N_0}(\mfg_{F_0})] \]
%are equal.  It should be possible to get an explicit upper bound on $\limsup_d \frac{T_j(d)}{d \log \log d}$ in terms of $N_0$ and $I(N_0)$.  Moreover, given an open subgroup $G \subset \GL_2(\widehat{\Z})/\{\pm 1\}$, it should be possible to explicitly compute $\limsup_d \frac{T_j(d)}{d \log \log d}$ for all $j$ such that the reduced adelic Galois representation has image $G$.   It would be interesting to do this for various elliptic curves $E_{/\Q}$.  
\end{enumerate}
\end{remark}

\section*{Acknowledgments} 
\noindent
We thank Abbey Bourdon and Drew Sutherland for helpful and stimulating conversations on the subject matter of this paper. The second author is supported by NSF award DMS-1402268.


\begin{thebibliography}{CCRS13}


\bibitem[BCP]{BCP} A. Bourdon, P.L. Clark and P. Pollack, \emph{Anatomy of torsion in the CM case}. 
To appear in Math. Z. 

\bibitem[BCS]{BCS16} A. Bourdon, P.L. Clark and J. Stankewicz, \emph{Torsion points on CM elliptic curves over real number fields}.  To appear in Transactions of the AMS.

\bibitem[BC]{BC17} A. Bourdon and P.L. Clark. \emph{Torsion points and Galois representations on CM elliptic curves}. 
\url{http://alpha.math.uga.edu/~pete/Bourdon_Clark.pdf}

\bibitem[BP]{BP16} A. Bourdon and P. Pollack. \emph{Torsion subgroups of CM
elliptic curves over odd degree number fields}. To appear in Int.  Math.
Res.  Notices.

\bibitem[Br10]{Breuer10} F. Breuer, \emph{Torsion bounds for elliptic curves and Drinfeld modules.}
J. Number Theory 130 (2010), 1241–-1250. 


\bibitem[CCRS13]{TORS1} P.L. Clark, B. Cook and J. Stankewicz, \emph{Torsion points on elliptic curves with complex multiplication (with an appendix by Alex Rice)}.  International Journal of Number Theory 9 (2013), 447--479.

\bibitem[CP15]{CP15} P.L. Clark and P. Pollack. \emph{The truth about torsion in the CM case}. C. R. Math. Acad. Sci. Paris \textbf{353} (2015), 683--688.

\bibitem[Dav00]{Davenport00} H. Davenport. \emph{Multiplicative number theory}. 3rd edition. Graduate Texts in Mathematics, no. 74. Springer-Verlag, New York, 2000. 

%\bibitem[HS99]{Hindry-Silverman99} M. Hindry and J. Silverman,
%\emph{Sur le nombre de points de torsion rationnels sur une courbe elliptique.}
%C. R. Acad. Sci. Paris S\'er. I Math. 329 (1999), no. 2, 97--100.


\bibitem[FT93]{FT93} A. Fr\"{o}hlich and M.J. Taylor, \emph{Algebraic number theory}. Cambridge Studies in Advanced Mathematics, no. 27. Cambridge University Press, Cambridge, 1993.


\bibitem[Gr10]{Greicius10} A. Greicius, \emph{Elliptic curves with surjective adelic Galois representations.} Experiment. Math. 19 (2010),  495-–507.

%\bibitem[Ja96]{Janusz96} G.J. Janusz. \emph{Algebraic number fields}. 2nd edition. Graduate Studies in Mathematics, no. 7.  American Mathematical Society, Providence, RI, 1996.

\bibitem[Me96]{Merel96} L. Merel, \emph{Bornes  pour  la  torsion  des  courbes  elliptiques  sur  les  corps  de  nombres}.
Invent. Math. 124 (1996), 437--449.

\bibitem[MPP]{MPP16} N. McNew, P. Pollack, and C. Pomerance. \emph{Numbers divisible by a large shifted prime and large torsion subgroups of CM elliptic curves.} To appear in Int. Math. Res. Notices.

\bibitem[HW08]{HW} G.H. Hardy and E.M. Wright, \emph{An introduction to the theory of numbers}. Sixth edition.
%Revised by D. R. Heath-Brown and J. H. Silverman. With a foreword by Andrew Wiles.
Oxford University Press, Oxford, 2008.

\bibitem[Ro99]{Rosen99} M. Rosen, \emph{A generalization of Mertens' Theorem}. J. Ramanujan Math. Soc. \textbf{14} (1999), 1–-19.


\bibitem[Se72]{Serre72} J.-P. Serre, \emph{Propri\'et\'es galoisiennes des points %d'ordre fini des courbes
elliptiques}. Invent. Math.  15  (1972), no. 4, 259--331.

\bibitem[Se97]{Serre-MW} J.-P. Serre, \emph{Lectures on the Mordell-Weil theorem.} 3rd edition. Aspects of Mathematics. Friedr. Vieweg \& Sohn, Braunschweig, 1997.

\bibitem[Si94]{SilvermanII} J. Silverman, \emph{Advanced Topics in the Arithmetic of Elliptic Curves}, Graduate Texts in
Mathematics, no. 151. Springer-Verlag, 1994.

\bibitem[Te15]{tenenbaum} G. Tenenbaum, \emph{Introduction to analytic and probabilistic number theory}. 3rd edition. Graduate Studies in Mathematics, no. 163. American Mathematical Society, Providence, RI, 2015.

\bibitem[Zy10]{Zywina10} D. Zywina, \emph{Elliptic curves with maximal Galois action on their torsion points}.
Bull. Lond. Math. Soc. 42 (2010),  811-–826.

\end{thebibliography}
\end{document}